\documentclass[11pt,bezier]{article}
\usepackage{graphicx}
\usepackage{float}
\usepackage{amsmath}
\usepackage{amsfonts,amsthm,amssymb}
\usepackage{graphics}
\usepackage{graphicx}

\newtheorem{lem}{Lemma}[section]
\newtheorem{thm}[lem]{Theorem}

\theoremstyle{definition}

\newcommand{\sq}{^{2}}

\newcommand{\inc}{\operatorname{inc}}

\begin{document}
	
	\title{On the maximally edge-connected square graphs with degree sequence conditions\footnote{The research is supported by Natural Science Foundation of Xinjiang Uygur Autonomous Region (2025D01E02) and National Natural Science Foundation of China (12261086).}}
	\author{Kai Li, Yingzhi Tian\footnote{Corresponding author. E-mail: tianyzhxj@163.com (Y. Tian).} \\
		{\small College of Mathematics and System Sciences, Xinjiang
			University, Urumqi, Xinjiang, 830046, PR China}}
	\date{}
	\maketitle
	
	\noindent{\bf Abstract}
Let $G=(V(G),E(G))$ be a finite connected simple graph. The edge-connectivity $\lambda(G)$ of $G$ is the minimum number of edges that must be removed to make $G$ disconnected. We call $G$ maximally edge-connected when $\lambda(G)=\delta(G)$, where $\delta(G)$ is the minimum degree of $G$.  The square graph $G^2$ of $G$ is the graph with vertex set $V(G)$, where two distinct vertices $u$ and $v$ are adjacent in $G^2$ precisely when their distance in $G$ is at most two. In this paper,  first we prove that if $\lambda(G)\geq4$ and  $\lambda(G^2)<\delta(G^2)$, then $\lambda(G^2)>\lambda(G)^{3/2}-\frac54\lambda(G)$, which extends the result in \cite{ChenXu} from the edge-connectivity range $4\leq\lambda(G)\leq7$ to $\lambda(G)\geq4$. Then we show that  $G^2$ is maximally edge-connected if one  of the following conditions hold: (i) $\delta(G)\geq\lfloor(n+2)/4\rfloor$, or $\delta(G)\leq\lfloor(n+2)/4\rfloor-1$ and $\sum_{i=1}^{2k}d_{n+1-i}\geq k(n-4)-2\delta(G)+3$, where $k$ is an integer satisfying $2\leq k\leq\delta(G)+2$ and $d_1\geq d_2\geq\cdots\geq d_n$ are the degrees of $G$; (ii) $d_{n-\delta(G)}+d_n\geq(3n-9)/4$.


	\section{Introduction}
	
In a communication network, stations and communication channels can be identified with vertices and edges, respectively. We therefore describe such a system by a finite simple graph $G=(V(G),E(G))$, where $V(G)$ and $E(G)$ represent the vertex set and the edge set of $G$, respectively. For every vertex $v\in V(G)$, its degree is denoted by $d_G(v)$. The minimum degree of $G$ is denoted by $\delta(G)$. For notation and terminology not introduced here, we follow \cite{BondyMurty}. The failure of several communication links may disconnect the network, so its resistance to link failures can be measured by the fewest links whose removal makes the network disconnected. The edge-connectivity  $\lambda(G)$ of $G$ is used to measure the reliability and fault-tolerance of corresponding network, namely, the least number of edges need to removal  to make $G$ disconnected. Whitney \cite{Whitney} showed that $\lambda(G)\leq\delta(G)$. So the graph $G$ satisfying $\lambda(G)=\delta(G)$ is called maximally edge-connected. Chartrand \cite{ChartrandCommunication} proved that a connected graph $G$ of order $n$ is maximally edge-connected if $\delta(G)\geq\lfloor n/2\rfloor$. A different degree condition was later given by Lesniak \cite{Lesniak}: whenever $d_G(u)+d_G(v)\geq n-1$ holds for every pair of distinct nonadjacent vertices $u$ and $v$, one has $\lambda(G)=\delta(G)$. Goldsmith and White \cite{GoldsmithWhite} derived another criterion involving disjoint vertex pairs. More precisely, if $G$ has $\lfloor n/2\rfloor$ mutually vertex-disjoint pairs $\{u_i,v_i\}$ satisfying $d_G(u_i)+d_G(v_i)\geq n$ for all $i=1,2,\ldots,\lfloor n/2\rfloor$, then $\lambda(G)=\delta(G)$. For more results, please see the survey by Hellwig and Volkmann \cite{HellwigVolkmannSurvey}.

In many networked structures, the interaction between two nodes may depend not only on a direct link, but also on whether a short path through intermediate nodes is available. For example, if at most one intermediate node is permitted, then two nodes whose distance in the original network is one or two can exchange information within two transmissions. This observation motivates an augmented network on the same node set, in which every pair of distinct nodes with distance at most two in the original network is connected. In graph-theoretic terminology, the graph obtained in this way is the square graph of $G$, written as $G^2$. Thus, $V(G^2)=V(G)$, and for any two distinct $u,v\in V(G)$, they are adjacent in $E(G^2)$ if and only if their distance in $G$ is at most two. Equivalently, $G^2$ retains all edges of $G$ and adds an edge between every nonadjacent pair that has a common neighbor in $G$. The added edges create additional communication routes and may therefore enhance local connectivity as well as resistance to link failures. The connectivity of square graphs has been studied in several recent works. Balbuena and Dankelmann \cite{BalbuenaDankelmann} proved that any connected graph $G$ of order $n$ satisfying $\delta(G)\geq\left\lfloor\frac{n+2}{4}\right\rfloor$ is maximally edge-connected. Moreover, they obtained a lower bound on $\lambda(G^2)$ in terms of $\lambda(G)$ when $G^2$ is not maximally edge-connected. Chen and Xu \cite{ChenXu} subsequently showed that if  the degree sum of any two distinct vertices of $G$ is at least $(n-1)/2$, then the square graph $G^2$ is maximally edge-connected. In addition, they improved the lower bound on $\lambda(G^2)$ in terms of $\lambda(G)$ in \cite{BalbuenaDankelmann}. More specifically, when $\lambda(G^2)<\delta(G^2)$, they proved that $\lambda(G^2)>\lambda(G)^{\frac32}-\frac54\lambda(G)$ for $4\leq\lambda(G)\leq7$, whereas for $\lambda(G)\geq8$ they obtained $\lambda(G^2)>\frac12\lambda(G)^{\frac32}+\frac12\lambda(G)-\lambda(G)^{\frac12}$. In this paper, we extend the condition $4\leq\lambda(G)\leq7$ to $\lambda(G)\geq4$ and
proved that whenever $\lambda(G)\geq4$ and $\lambda(G^2)<\delta(G^2)$, one has $\lambda(G^2)>\lambda(G)^{\frac32}-\frac54\lambda(G)$.
	
Conditions involving the smaller terms of a degree sequence have also been used to guarantee a graph to be maximally edge-connected. Consider a connected graph $G$ with $n$ vertices, whose degrees are indexed so that $d_1\geq d_2\geq\cdots\geq d_n$.  Dankelmann and Volkmann \cite{DankelmannVolkmannDegree} obtained $\lambda(G)=\delta(G)$ in either of the following two situations: $\delta(G)\geq\lfloor n/2\rfloor$, or $\delta(G)\leq\lfloor n/2\rfloor-1$ and there is an integer $k$ with $2\leq k\leq\delta(G)$ for which $\sum_{i=1}^{2k}d_{n+1-i}\geq kn-3$. Hellwig and Volkmann \cite{HellwigVolkmann2003} later improved the second case by replacing this degree-sum bound with $\sum_{i=1}^{2k}d_{n+1-i}\geq\max\{k(n-1)-1,(k-1)n+2\delta(G)-1\}$, while keeping the same ranges for $\delta(G)$ and $k$. In this paper, we derive analogous degree sequence conditions for $G^2$ to be maximally edge-connected. We prove that  $G^2$ is maximally edge-connected  whenever either $\delta(G)\geq\lfloor(n+2)/4\rfloor$, or $\delta(G)\leq\lfloor(n+2)/4\rfloor-1$ and there is an integer $k$ from the range $2\leq k\leq\delta(G)+2$ such that  $\sum_{i=1}^{2k}d_{n+1-i}\geq k(n-4)-2\delta(G)+3$. In addition, we prove that  $G^2$ is maximally edge-connected if $\delta(G)\geq2$ and $d_{n-\delta(G)}+d_n\geq(3n-9)/4$.
	
The subsequent sections are arranged as follows.    Section~2 sets up the notation and terminology and records the auxiliary results. Section~3 develops a lower bound for $\lambda(G^2)$ in terms of $\lambda(G)$. Section~4 is devoted to degree sequence conditions under which $\lambda(G^2)=\delta(G^2)$. The final section collects the principal conclusions and discusses several possible directions for future work.

	\section{Preliminaries}

Throughout this paper, $G=(V(G),E(G))$ represents a finite simple graph, where $V(G)$ and $E(G)$ refer to the vertex set and the edge set of $G$, respectively. The order of $G$ is given by $n(G)=|V(G)|$. The expressions $pq$ and $qp$ represent the same edge with endvertices $p$ and $q$. For disjoint vertex subsets $U,W\subseteq V(G)$, the collect of edges between them is denoted by $E_G(U,W)=\{pq\in E(G):p\in U,\ q\in W\}$. For $z\in V(G)$,  its neighborhood is $N_G(z)=\{w\in V(G):zw\in E(G)\}$ and its closed neighborhood is $N_G[z]=N_G(z)\cup\{z\}$. The degree $d_G(z)$ of $z$ is the cardinality $|N_G(z)|$. We use $N_U(z)=N_G(z)\cap U$ for the neighbors of $z$ belonging to $U$. Arranging the degrees of all vertices from largest to smallest produces the degree sequence of $G$.
	
Fix a minimum edge-cut $S\subseteq E(G^2)$. Write $H_1,H_2$ for the two components of $G^2-S$. Let $V_i=V(H_i)$ and $G_i=G[V_i]$ for $i=1,2$. Restricting $S$ to the original edges of $G$ gives $S'=S\cap E(G)$. The deletion of $S'$ leaves a spanning subgraph of $G^2-S$; hence $G-S'$ is disconnected. Its components lie entirely within one of the two vertex sets, and those contained in $V_i$ together form $G_i$. To distinguish boundary vertices from interior vertices, let $A$ denote the endvertices in $V_1$ of the edges in $S$, and let $B$ denote their endvertices in $V_2$. Vertices outside these boundary sets are termed interior vertices, giving $A_0=V_1-A$ and $B_0=V_2-B$ for $H_1$ and $H_2$, respectively. For the edge-cut $S'$ of $G$, we use $A'$ and $B'$ to record its endvertices in $G_1$ and $G_2$, respectively. The inclusion $S'\subseteq S$ then yields $A'\subseteq A$ and $B'\subseteq B$.

The subsequent proofs draw on the following four auxiliary results.
  
 \begin{thm}[\cite{ChartrandCommunication}]
 	If the minimum degree of a graph $G$ on $n$ vertices is at least $\frac{n-1}{2}$,
 	 then $G$ is maximally edge-connected.
 \end{thm}
  
  \begin{thm}[\cite{BalbuenaDankelmann}]
  	If the  minimum degree of a connected graph $G$ on $n$ vertices is at least
  	 $\left\lfloor\frac{n+2}{4}\right\rfloor$, then
  	 $G^2$ is maximally edge-connected.
  \end{thm}
  
  \begin{lem}[{\cite{Hamidoune}}]
  	If $S$ is a minimum edge-cut of a connected graph $G$ and $|S|<\delta(G)$, then every component of $G-S$ contains at least one interior vertex.
  \end{lem}
  
\begin{lem}[\cite{DankelmannVolkmannDegree}]
	Let $S$ be a minimum edge-cut of a connected graph $G$. If $\lambda(G)<\delta(G)$, then every component of $G-S$ contains at least $\delta(G)+1$ vertices.
\end{lem}

	\section{A lower bound on $\lambda(G\sq)$ in terms of $\lambda(G)$}
	
   \begin{thm}
	If $G$ is a connected graph with $\lambda(G)\geq4$, then either $\lambda(G\sq)=\delta(G\sq)$ or
	$\lambda(G\sq)>\lambda(G)^{\frac32}-\frac54\lambda(G)$.
   \end{thm}
	
	\noindent{\bf Proof.}
	Write $\lambda=\lambda(G)$. Assume $\lambda(G^2)<\delta(G^2)$. Choose an edge-cut $S$ of $G^2$ with $|S|=\lambda(G^2)$. We continue to use the notation $H_1$, $H_2$, $V_1$, $V_2$, $S'$, $A'$, and $B'$ defined in Section~2. For each $u\in A'\cup B'$, set $\operatorname{inc}(u,S')=\bigl|\{e\in S':u\text{ is an endvertex of }e\}\bigr|$.
	
	\noindent{\bf Case 1.}
	There exists a vertex $u\in A'\cup B'$ such that $\sqrt{\lambda}\leq\inc(u,S')\leq\lambda-\sqrt{\lambda}$. 
	
	 By symmetry, we only need to consider the case $u\in A'$. Choose $v\in N_G(u)\cap V_1$ and $w\in N_G(u)\cap V_2$. We have $u\in N_G(v)\cap N_G(w)$. Therefore, they are adjacent in the square graph $G\sq$, that is, $vw\in E(G\sq)$. Since $v$ and $w$ lie in different components of $G^2-S$ and $vw\in E(G^2)$, we have $vw\in S$. Therefore,
	\[
	\begin{aligned}
		|S|&\geq |N_G(u)\cap V_1|\,|N_G(u)\cap V_2|+\inc(u,S')\\
		&=(d_G(u)-\inc(u,S'))\inc(u,S')+\inc(u,S')\\
		&\geq (\lambda-\inc(u,S'))\inc(u,S')+\inc(u,S').
	\end{aligned}
	\]
	For $\sqrt{\lambda}\leq x\leq\lambda-\sqrt{\lambda}$, consider $f(x)=(\lambda-x)x+x$. Its minimum on this interval occurs at $x=\sqrt{\lambda}$. We therefore obtain
	$|S|\geq\lambda^{\frac32}-\lambda+\sqrt{\lambda}>\lambda^{\frac32}-\frac54\lambda$.
	
	\noindent{\bf Case 2.}
	For every $v\in A'\cup B'$, either
	$\inc(v,S')<\sqrt{\lambda}$ or
	$\inc(v,S')>\lambda-\sqrt{\lambda}$.
	
	Partition $A'$ according to these two alternatives: let $A^+$ consist of the vertices $v\in A'$ for which $\inc(v,S')>\lambda-\sqrt{\lambda}$, and let $A^-$ consist of those satisfying $\inc(v,S')<\sqrt{\lambda}$. Define $B^+$ and $B^-$ in the same way. Hence $A^+\cup A^-=A'$ and $B^+\cup B^-=B'$. Put
	\[
	W_1=(V_1-A^+)\cup B^+,
	\qquad
	W_2=(V_2-B^+)\cup A^+,
	\qquad
	T=E_G(W_1,W_2).
	\]
	Because no edge of $G$ joins $V_1-A'$ to $V_2-B'$, we have
	$E_G(V_1-A^+,V_2-B^+)=E_G(A^-,B^-)$. Hence
	\[
	T=E_G(A^+,B^+)\cup E_G(A^-,B^-)\cup E_G(A^+,V_1-A^+)\cup E_G(B^+,V_2-B^+).
	\]
	The set $T$ disconnects $G$, so $|T|\geq\lambda$. Therefore, at least one of
	$|E_G(A^+,B^+)|$, $|E_G(A^-,B^-)|$,
	$|E_G(A^+,V_1-A^+)|$ and
	$|E_G(B^+,V_2-B^+)|$
	is at least $\lambda/4$.
	
	\noindent{\bf Subcase 2.1.}
	$|E_G(A^+,B^+)|\geq\lambda/4$.
	
	Since $|E_G(A^+,B^+)|\leq|A^+||B^+|$, this assumption gives
	$|A^+||B^+|\geq\lambda/4$. It follows that
	$|A^+|+|B^+|\geq\sqrt{\lambda}$. By symmetry, suppose that
	$|A^+|\geq|B^+|$. In particular,
	$|A^+|\geq\sqrt{\lambda}/2$. Choose $w'\in B^+$ so that
	\[
	\inc(w',S')=\min\{\inc(w,S'):w\in B^+\}.
	\]
	
	\noindent\textbf{Subcase 2.1.1.}
	$|A^+|<\inc(w',S')$.
	
	Consider first the situation $|A^+|<\sqrt{\lambda}$. Then
	\[
	\begin{aligned}
		|S|
		&\geq \sum_{v\in A^+}\inc(v,S')+\sum_{w\in B^+}\bigl(\inc(w,S')-|A^+|\bigr)\\
		&>|A^+|(\lambda-\sqrt{\lambda})+|B^+|\bigl(\inc(w',S')-|A^+|\bigr)\\
		&>|A^+|(\lambda-\sqrt{\lambda})+(\sqrt{\lambda}-|A^+|)(\lambda-\sqrt{\lambda}-|A^+|)\\
		&=\lambda^{\frac32}-\lambda-|A^+|\sqrt{\lambda}+|A^+|^2\\
		&=\lambda^{\frac32}+\left(|A^+|-\frac{\sqrt{\lambda}}{2}\right)^2
		-\frac54\lambda\\
		&\geq \lambda^{\frac32}-\frac54\lambda.
	\end{aligned}
	\]
	Assume next that $|A^+|\geq\sqrt{\lambda}$. Using only the edges of $S'$ incident with vertices of $A^+$ gives
	\[
	|S|\geq\sum\limits_{v\in A^+}\operatorname{inc}(v,S')
	>|A^+|(\lambda-\sqrt{\lambda})
	\geq\sqrt{\lambda}(\lambda-\sqrt{\lambda})
	=\lambda^{\frac32}-\lambda
	>\lambda^{\frac32}-\frac54\lambda.
	\]
	
	\noindent\textbf{Subcase 2.1.2.}
	$|A^+|\geq\inc(w',S')$.
	
	By the choice of $w'$, we have
	$\inc(w',S')>\lambda-\sqrt{\lambda}$, and hence
	$|A^+|>\lambda-\sqrt{\lambda}$. It follows that
	\[
	|S|\geq\sum_{v\in A^+}\operatorname{inc}(v,S')
	>|A^+|(\lambda-\sqrt{\lambda})
	>(\lambda-\sqrt{\lambda})^2
	\geq\sqrt{\lambda}(\lambda-\sqrt{\lambda})
	=\lambda^{\frac32}-\lambda
	>\lambda^{\frac32}-\frac54\lambda.
	\]
	
	\noindent{\bf Subcase 2.2.}
	$|E_G(A^-,B^-)|\geq\lambda/4$.
	
	Let $A^*$ be the set of vertices of $A^-$ having at least one neighbor in $B^-$, and define $B^*$ symmetrically. Then
	\[
	|E_G(A^*,B^*)|=|E_G(A^-,B^-)|\geq\frac{\lambda}{4}.
	\]
	Thus
	$|A^*||B^*|\geq|E_G(A^*,B^*)|\geq\lambda/4$. By the arithmetic--geometric mean inequality,
	$|A^*|+|B^*|\geq2\sqrt{|A^*||B^*|}\geq\sqrt{\lambda}$. Interchanging $A^*$ and $B^*$ if necessary, assume
	$|A^*|\geq|B^*|$. Hence
	$|A^*|\geq\frac{|A^*|+|B^*|}{2}\geq\frac{\sqrt{\lambda}}{2}$.
	
	For each $v\in A^*$,
	$|N_{V_1}(v)|=d_G(v)-\inc(v,S')>\lambda-\sqrt{\lambda}$.
	Likewise, every $w\in B^*$ satisfies
	$|N_{V_2}(w)|>\lambda-\sqrt{\lambda}$.
	Choose $x'\in A^*$ for which $|N_{V_1}(x')|$ is minimum; equivalently,
	\[
	|N_{V_1}(x')|=\min\{|N_{V_1}(x)|:x\in A^*\}.
	\]
	
	\noindent\textbf{Subcase 2.2.1.}
	$|A^*|<|N_{V_1}(x')|$.
	
	First consider $|A^*|<\sqrt{\lambda}$. We then have
	\[
	\begin{aligned}
		|S|&\geq \sum_{\substack{v\in A^*\\ w\in N_{V_2}(v)\cap B^*}}|N_{V_2}(w)|+\sum_{\substack{w\in B^*\\ v\in N_{V_1}(w)\cap A^*}}(|N_{V_1}(v)|-|A^*|\bigr)\\
		&>|A^*|(\lambda-\sqrt{\lambda})
		+|B^*|\bigl(|N_{V_1}(x')|-|A^*|\bigr)\\
		&>|A^*|(\lambda-\sqrt{\lambda})
		+(\sqrt{\lambda}-|A^*|)(\lambda-\sqrt{\lambda}-|A^*|)\\
		&=\lambda^{\frac32}-\lambda-|A^*|\sqrt{\lambda}+|A^*|^2\\
		&=\lambda^{\frac32}
		+\left(|A^*|-\frac{\sqrt{\lambda}}{2}\right)^2-\frac54\lambda\\
		&\geq\lambda^{\frac32}-\frac54\lambda.
	\end{aligned}
	\]
	Now let $|A^*|\geq\sqrt{\lambda}$. The first family of crossing edges alone yields
	\[
	|S|\geq\sum_{\substack{v\in A^*\\ w\in N_{V_2}(v)\cap B^*}}|N_{V_2}(w)|
	>|A^*|(\lambda-\sqrt{\lambda})
	\geq\sqrt{\lambda}(\lambda-\sqrt{\lambda})
	=\lambda^{\frac32}-\lambda
	>\lambda^{\frac32}-\frac54\lambda.
	\]
	
	\noindent\textbf{Subcase 2.2.2.}
	$|A^*|\geq|N_{V_1}(x')|$.
	
	Since $|N_{V_1}(x')|>\lambda-\sqrt{\lambda}$, this subcase gives
	$|A^*|>\lambda-\sqrt{\lambda}$. Consequently,
	\[
	|S|\geq\sum_{\substack{v\in A^*\\ w\in N_{V_2}(v)\cap B^*}}|N_{V_2}(w)|
	>|A^*|(\lambda-\sqrt{\lambda})
	>(\lambda-\sqrt{\lambda})^2
	\geq\sqrt{\lambda}(\lambda-\sqrt{\lambda})
	=\lambda^{\frac32}-\lambda
	>\lambda^{\frac32}-\frac54\lambda.
	\]
	
	\noindent{\bf Subcase 2.3.}
	$|E_G(A^+,V_1-A^+)|\geq\lambda/4$ or
	$|E_G(B^+,V_2-B^+)|\geq\lambda/4$.
	
	Suppose first that
	$|E_G(A^+,V_1-A^+)|\geq\lambda/4$.
	Let $V'_1$ consist of those vertices in $V_1-A^+$ that have a neighbor in $A^+$. Then
	\[
	|E_G(A^+,V'_1)|=|E_G(A^+,V_1-A^+)|\geq \frac{\lambda}{4},
	\]
	and hence $|A^+|+|V'_1|\geq\sqrt{\lambda}$. Choose $y'\in A^+$ so that
	\[
	\inc(y',S')=\min\{\inc(y,S'):y\in A^+\}.
	\]
	Therefore,
	\[
	\begin{aligned}
		|S|
		&\geq \sum_{y\in A^+}\inc(y,S')+|V'_1|\inc(y',S')\\
		&\geq |A^+|\inc(y',S')+|V'_1|\inc(y',S')\\
		&>(|A^+|+|V'_1|)(\lambda-\sqrt{\lambda})\\
		&\geq \sqrt{\lambda}(\lambda-\sqrt{\lambda})\\
		&=\lambda^{\frac32}-\lambda\\
		&>\lambda^{\frac32}-\frac54\lambda.
	\end{aligned}
	\]
	
	The alternative
	$|E_G(B^+,V_2-B^+)|\geq\lambda/4$
	is obtained by the same argument with the two sides interchanged. Thus every case gives
	$|S|>\lambda^{\frac32}-\frac54\lambda$.
	The proof is complete. $\Box$
	
	\section{Degree sequence conditions for the square graph $G\sq$ to be maximally edge-connected}
	
	\begin{lem}\label{lem:interior-vertices-square}
	
		Let $G$ be a connected graph with $\lambda(G\sq)<\delta(G\sq)$. Assume $S$ is a minimum edge-cut of $G^2$, and denote the two components of $G^2-S$ by $H_1$ and $H_2$. Let $A_0$ and $B_0$ be the vertices of $H_1$ and $H_2$ that are not incident with any edge of $S$, respectively. Then
		$|A_0|\geq\delta(G)+1$ and $|B_0|\geq\delta(G)+1$.

	\end{lem}
	
	\noindent{\bf Proof.}
	Let $V_i=V(H_i)$ for $i=1,2$. By Lemma~2.4, $|V_i|\geq\delta(G\sq)+1\geq |S|+2$ for $i=1,2$.  By Lemma~2.3,  $A_0$ and $B_0$ are both nonempty.
	
	Suppose $|A_0|\leq\delta(G)$. Take a shortest path $R$ in $G$ from $B_0$ to $V_1$, and let $w\in V_2$ be the vertex immediately preceding the first vertex of this path $R$ in $V_1$. Since no vertex of $B_0$ has a neighbor in $V_1$ in $G\sq$, the vertex $w$ has neighbors in both $V_1$ and $V_2$ in $G$. Set
	\[
	r=|N_G(w)\cap V_1|\geq1,
	\qquad
	t=|N_G(w)\cap V_2|\geq1.
	\]
	Then
	\[
	rt\geq r+t-1=d_G(w)-1\geq\delta(G)-1.
	\]
	
	For each $v\in N_G(w)\cap V_1$, the vertex $v$ is adjacent in $G\sq$ to $w$ and to every vertex of $N_G(w)\cap V_2$, and hence has at least $t+1$ incident edges belonging to $S$. Each of the remaining $|V_1|-|A_0|-r$ boundary vertices is incident with at least one edge of $S$. Therefore, counting the edges of $S$ by their endpoints in $V_1$, we obtain
	\[
	\begin{aligned}
		|S|
		&\geq r(t+1)+\bigl(|V_1|-|A_0|-r\bigr)\\
		&=|V_1|-|A_0|+rt\\
		&\geq |V_1|-|A_0|+\delta(G)-1\\
		&\geq |V_1|-1\\
		&\geq |S|+1,
	\end{aligned}
	\]
	where the penultimate inequality follows from $|A_0|\leq\delta(G)$. This is a contradiction. Hence $|A_0|\geq\delta(G)+1$. By symmetry, $|B_0|\geq\delta(G)+1$. $\Box$

     \begin{thm}\label{thm:degree-sequence-square}
     	Let $G$ be a connected graph on $n$ vertices with degree sequence
     	$d_1\geq d_2\geq\cdots\geq d_n$. Then $G^2$ is
     	maximally edge-connected provided that either
     	$\delta(G)\geq\lfloor(n+2)/4\rfloor$, or
     	$\delta(G)\leq\lfloor(n+2)/4\rfloor-1$ and
     	$\sum\limits_{i=1}^{2k}d_{n+1-i}\geq k(n-4)-2\delta(G)+3$
     	for some integer $k$ satisfying $2\leq k\leq\delta(G)+2$.
     \end{thm}
    
    \noindent{\bf Proof.}
By Theorem~2.2,   $G^2$ is maximally edge-connected if
$\delta(G)\geq\lfloor(n+2)/4\rfloor$.
So we assume that
$\delta(G)\leq\lfloor(n+2)/4\rfloor-1$ and there is an integer $k$ with
$2\leq k\leq\delta(G)+2$ for which
$\sum\limits_{i=1}^{2k}d_{n+1-i}\geq k(n-4)-2\delta(G)+3$.
On the contrary, assume $\lambda(G^2)<\delta(G^2)$. Let $S$ be a
minimum edge-cut of $G^2$. Use $H_1$, $H_2$, $V_1$, $V_2$,
$A$, $B$, $A_0$, $B_0$, $S'$, $A'$ and $B'$ for the notations
associated with the edge-cut $S$ according to Section~2.
Lemma~\ref{lem:interior-vertices-square} implies
$|A_0|,|B_0|\geq\delta(G)+1$. By $G$ is connected, there is an edge $uv\in E_G(V_1,V_2)$ with $u\in A'$ and $v\in B'$.
	
	Choose $W\subseteq A_0$ and $T\subseteq B_0$ such that
	$|W|=|T|=k-1$. For every $x\in A_0$ and $y\in B_0$, $d_G(x)\leq |V_1|-|A'|-1$,
	$d_G(y)\leq |V_2|-|B'|-1$.
	Moreover, $d_G(u)\leq |V_1|-|A_0|-1+|B'|$, $d_G(v)\leq |V_2|-|B_0|-1+|A'|$.
	
	Let $Q=W\cup T\cup\{u,v\}$.
	Then $|Q|=2k$ and 
	\[
	\begin{aligned}
		\sum_{z\in Q}d_G(z)
		&\leq
		(k-1)(|V_1|-|A'|-1)
		+(k-1)(|V_2|-|B'|-1)\\
		&\quad
		+|V_1|-|A_0|-1+|B'|
		+|V_2|-|B_0|-1+|A'|\\
		&=
		kn-|A_0|-|B_0|-2k
		+(2-k)(|A'|+|B'|).
	\end{aligned}
	\]
	Since $|A_0|+|B_0|\geq2\delta(G)+2$,
	$|A'|+|B'|\geq2$ 
	and $k\geq2$, we obtain
	\[
	\sum_{z\in Q}d_G(z)\leq k(n-4)-2\delta(G)+2.
	\]
	On the other hand, since $Q$ contains $2k$ vertices,
	\[
	\sum_{z\in Q}d_G(z)
	\geq
	\sum\limits_{i=1}^{2k}d_{n+1-i}
	\geq
	k(n-4)-2\delta(G)+3,
	\]
	a contradiction. Therefore, $\lambda(G^2)=\delta(G^2)$, that is, $G^2$ is
     	maximally edge-connected.
	\hfill$\Box$

	The bound in Theorem~4.2 is sharp. Fix an integer $m\geq2$ and define $G_m=K_{m+1}+K_{m-1}+K_1+K_1+K_{m-1}+K_{m+1}$, where  ``$+$'' denotes the sequential sum: consecutive terms are joined completely, whereas no edges are placed between nonconsecutive terms. The graph $G_m$ has order $n=4m+2$ and minimum degree $\delta(G_m)=m$. Among its vertices, two have degree $m$, $2m+2$ have degree $2m-1$, and the remaining $2m-2$ have degree $2m$. Hence, whenever $2\leq k\leq m+2=\delta(G_m)+2$, the $2k$ smallest degrees consist of the two degrees equal to $m$ together with $2k-2$ degrees equal to $2m-1$. Therefore, $\sum\limits_{i=1}^{2k}d_{n+1-i}=2m+(2k-2)(2m-1)=k(n-4)-2\delta(G_m)+2$. On the other hand, although $\delta(G_m^2)=2m$, there is an edge cut of $G_m^2$ containing only $2m-1$ edges. It follows that $\lambda(G_m^2)\leq2m-1<2m=\delta(G_m^2)$, and consequently $G_m^2$ is not maximally edge-connected. Thus, the degree-sequence bound in Theorem~4.2 cannot be weakened by replacing the term $+3$ with $+2$.

	\begin{thm}\label{thm:degree-sequence-square}
		Let $G$ be a connected graph of order $n$ with degree sequence
		$d_1\geq d_2\geq\cdots\geq d_n$. If $\delta=\delta(G)\geq2$ and
		$d_{n-\delta}+d_n\geq(3n-9)/4$, then
		$G^2$ is maximally edge-connected.
	\end{thm}
	
	\noindent{\bf Proof.}  
	Assume, on the contrary, that $\lambda(G^2)<\delta(G^2)$. Choose a minimum edge-cut $S$ of $G^2$, and retain the notations $H_1$, $H_2$, $V_1$, $V_2$, $A$, $B$, $A_0$, $B_0$, $S'$, $A'$ and $B'$ be defined as in Section~2. By interchanging the two components if necessary, we may suppose that $|V_1|\leq n/2$ and $|V_2|\geq n/2$.

	By Lemma~2.3, the component $H_1$ contains an interior
	vertex $t$. Hence $t\notin A$ and, consequently, $t\notin A'$.
	Moreover, the definition of $A'$ implies that
	$N_G[t]\subseteq V_1\setminus A'$. Among all paths in $G$ joining $t$ to a vertex of $A'$, choose one of minimum length and denote it by $R$. Let $u\in A'$ be the end of $R$ contained in $A'$, and let $y$ be the vertex immediately preceding $u$ along $R$. By the minimality of $R$, we have
	$y\in V_1- A'$. Therefore,
	$y\in N_G(u)\cap(V_1-A')$, and hence
	$|N_G(u)\cap(V_1-A')|\geq1$.
	Repeating the above reasoning for $V_2-B'$, we can choose a vertex 
	$v\in B'$ with 
	$|N_G(v)\cap(V_2-B')|\geq1$.
	
	Choose a vertex $w\in A_0$ with maximum degree in $G$ among all vertices of $A_0$. Since
	$|A_0|\geq\delta(G)+1$, the ordering of the degree sequence
	implies that $d_G(w)\geq d_{n-\delta}$.   Moreover, because
	$w\in A_0$, all neighbors of $w$ in $G$ lie in $V_1-A'$. Since $G$ is connected, $E_G(V_1,V_2)\neq\varnothing$, and hence  $|A'|\geq1$.
	Consequently,
	$d_G(w)\leq |V_1|-|A'|-1\leq n/2-2$.
	
	Since $d_G(u),d_G(v)\geq d_n=\delta$ and
	$d_G(w)\geq d_{n-\delta}$, the hypothesis of the theorem gives
	\begin{align*}
		d_G(u)+d_G(v)+2d_G(w)
		&\geq 2d_n+2d_{n-\delta}\\
		&=2\bigl(d_n+d_{n-\delta}\bigr)\\
		&\geq\frac{3n-9}{2}.
	\end{align*}
	Using $d_G(w)\leq n/2-2$, we therefore obtain
	\begin{align*}
		d_G(u)+d_G(v)
		&\geq\frac{3n-9}{2}-2d_G(w)\\
		&\geq\frac{3n-9}{2}-(n-4)\\
		&=\frac{n-1}{2}.
	\end{align*}

	\medskip
	\noindent
	\textbf{Claim 1.}
	$|S|<d_G(u)+d_G(v)-1$.
	
	\smallskip
	
	Suppose otherwise that $|S|\geq d_G(u)+d_G(v)-1\geq(n-1)/2-1$. Since
	$|S|=\lambda(G^2)<\delta(G^2)$ and both quantities are integers, we
	have
	$\delta(G^2)\geq |S|+1\geq(n-1)/2$. Applying Theorem~2.1 to $G^2$,
	we obtain $\lambda(G^2)=\delta(G^2)$, contrary to our assumption.
	This proves Claim~1.

	We now partition the relevant neighborhoods of $u$ and $v$. Define
	\begin{align*}
		X_1&=N_G(u)\cap N_G(v)\cap V_1, &
		X_2&=N_G(u)\cap N_G(v)\cap V_2,\\
		Y_1&=\bigl(N_G(u)\cap B'\bigr)\setminus X_2, &
		Y_2&=\bigl(N_G(v)\cap A'\bigr)\setminus X_1,\\
		Z_1&=\bigl(N_G(u)\cap V_1\bigr)\setminus X_1, &
		Z_2&=\bigl(N_G(v)\cap V_2\bigr)\setminus X_2.
	\end{align*}
	For $i=1,2$, write
	$|X_i|=x_i$, $|Y_i|=y_i$, and $|Z_i|=z_i$.
	By the choices of $u$ and $v$, we have $z_1\geq1$ and $z_2\geq1$.
	
	We distinguish two cases according to whether $u$ and $v$ are
	adjacent in $G$.
	
	\medskip
	\noindent
	\textbf{Case 1.}
	$uv\in E(G)$.
	
	\smallskip
	
	Since $u\in A'\subseteq V_1$ and $v\in B'\subseteq V_2$, the edge
	$uv$ implies that $v\in Y_1$ and $u\in Y_2$. In particular,
	$y_1\geq1$ and $y_2\geq1$.
	
	\smallskip
	\noindent
	\textbf{Subcase 1.1.}
	$x_1+x_2=0$.
	
	\begin{figure}[H]
		\centering
		\includegraphics[width=0.5\textwidth]{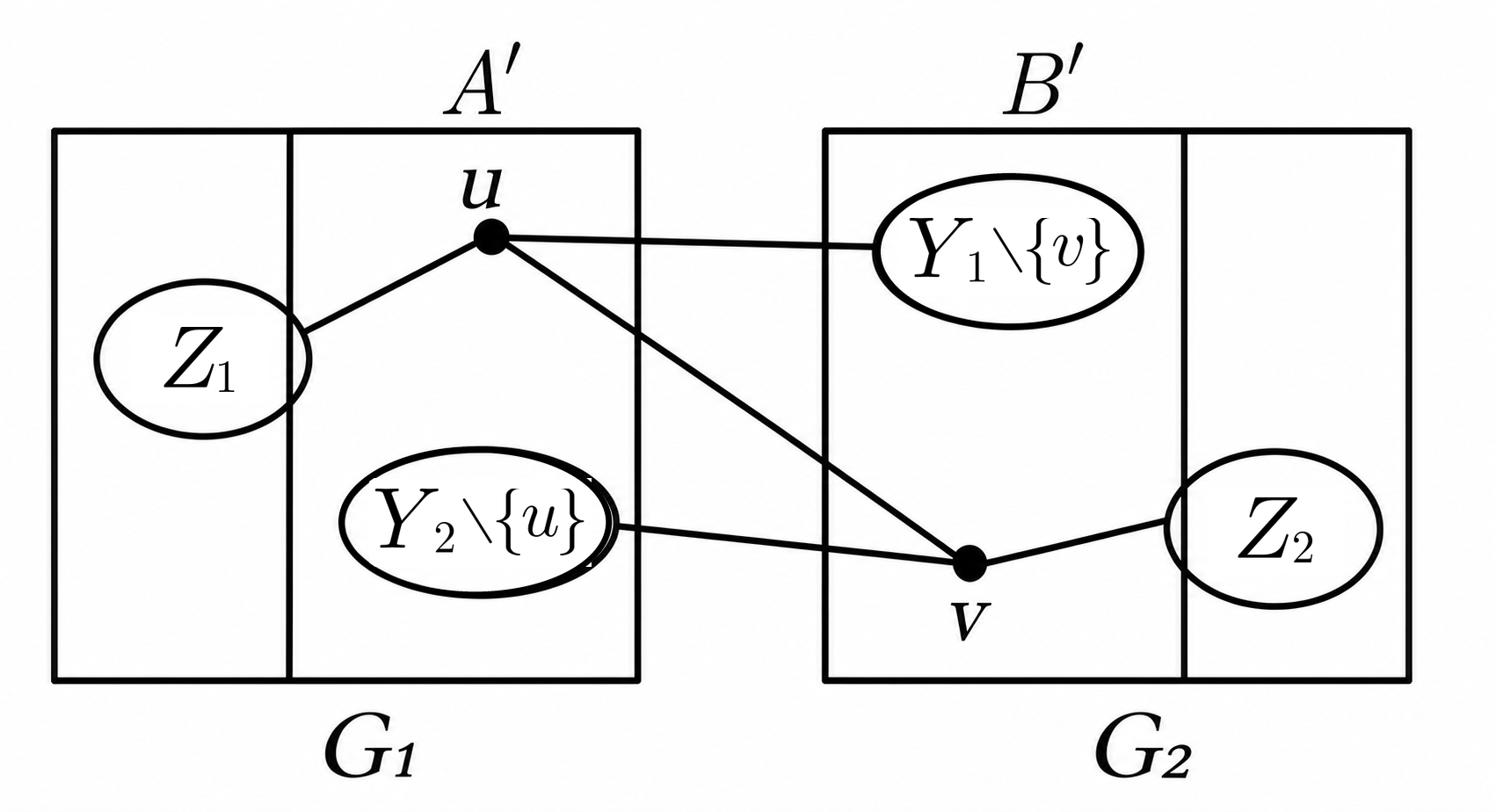}
  \caption{The illustration of $uv\in E(G)$ and $x_1+x_2=0$}
		\label{fig:figure1.1}
	\end{figure}
	
	In this subcase,
	$d_G(u)+d_G(v)=y_1+y_2+z_1+z_2$.
	Counting the edges of $G^2$ crossing from $V_1$ to $V_2$ between the
	sets defined above, and subtracting the edge $uv$, which is otherwise
	counted twice, gives
	$|S|\geq y_1+y_2+z_1y_1+z_2y_2-1$.
	Consequently,
	\begin{align*}
		|S|-d_G(u)-d_G(v)+1
		&\geq y_1+y_2+z_1y_1+z_2y_2-1-z_1-z_2-y_1-y_2+1\\
		&=z_1(y_1-1)+z_2(y_2-1)\\
		&\geq0.
	\end{align*}
	Thus, $|S|\geq d_G(u)+d_G(v)-1$, contradicting Claim~1.
	
	\smallskip
	\noindent
	\textbf{Subcase 1.2.}
	$x_1+x_2>0$.
	
	\begin{figure}[H]
		\centering
		\includegraphics[width=0.5\textwidth]{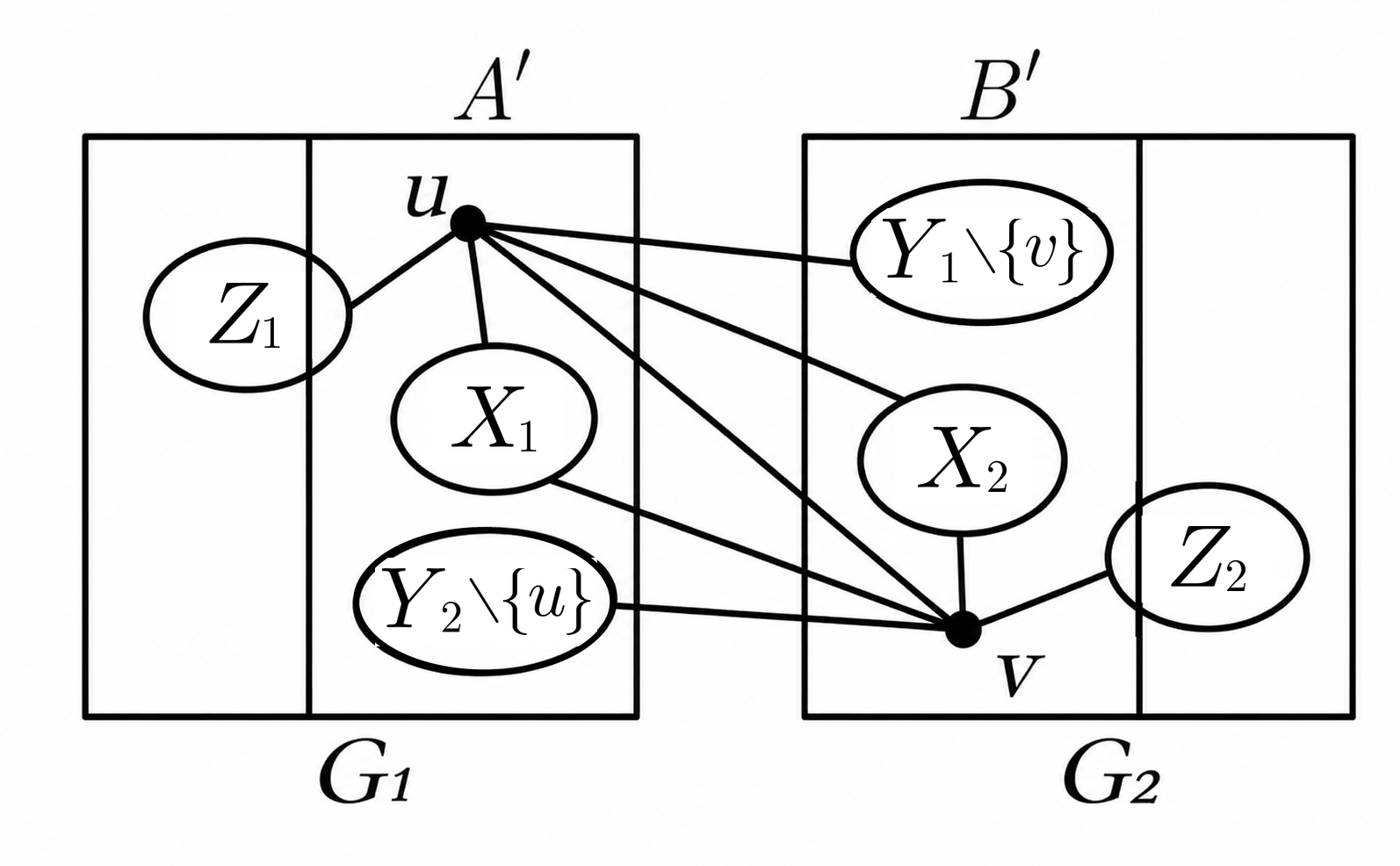}
  \caption{The illustration of $uv\in E(G)$ and $x_1+x_2>0$}
		\label{fig:figure1.2}
	\end{figure}
	
	Here,
	$d_G(u)+d_G(v)=2x_1+2x_2+y_1+y_2+z_1+z_2$.
	Counting the relevant pairwise disjoint families of crossing edges in
	$G^2$ yields
	\begin{align*}
		|S|\geq{}&
		z_1(y_1+x_2)+x_1(x_2+y_1+z_2)
		+(y_2-1)(1+x_2+z_2)+y_1+x_2+z_2\\
		={}&
		y_1+y_2+z_1y_1+z_1x_2+x_1x_2+x_1z_2
		+y_2x_2+z_2y_2+x_1y_1-1.
	\end{align*}
	
	After subtracting $d_G(u)+d_G(v)-1$ and collecting terms, we obtain
	\begin{align*}
		|S|-d_G(u)-d_G(v)+1
		\geq{}&
		(y_2-1)z_2+(y_1-1)z_1+(z_2+y_1-2)x_1+(z_1+y_2-2)x_2+x_1x_2.
	\end{align*}
	Since $y_1,y_2,z_1,z_2\geq1$, the right-hand side is a sum of nonnegative terms. Thus, $|S|\geq d_G(u)+d_G(v)-1$, also contradicting Claim~1.
	
	\medskip
	\noindent
	\textbf{Case 2.}
	$uv\notin E(G)$.
	
	\smallskip
	\noindent
	\textbf{Subcase 2.1.}
	$x_1+x_2=0$.
	
	\begin{figure}[H]
		\centering
		\includegraphics[width=0.5\textwidth]{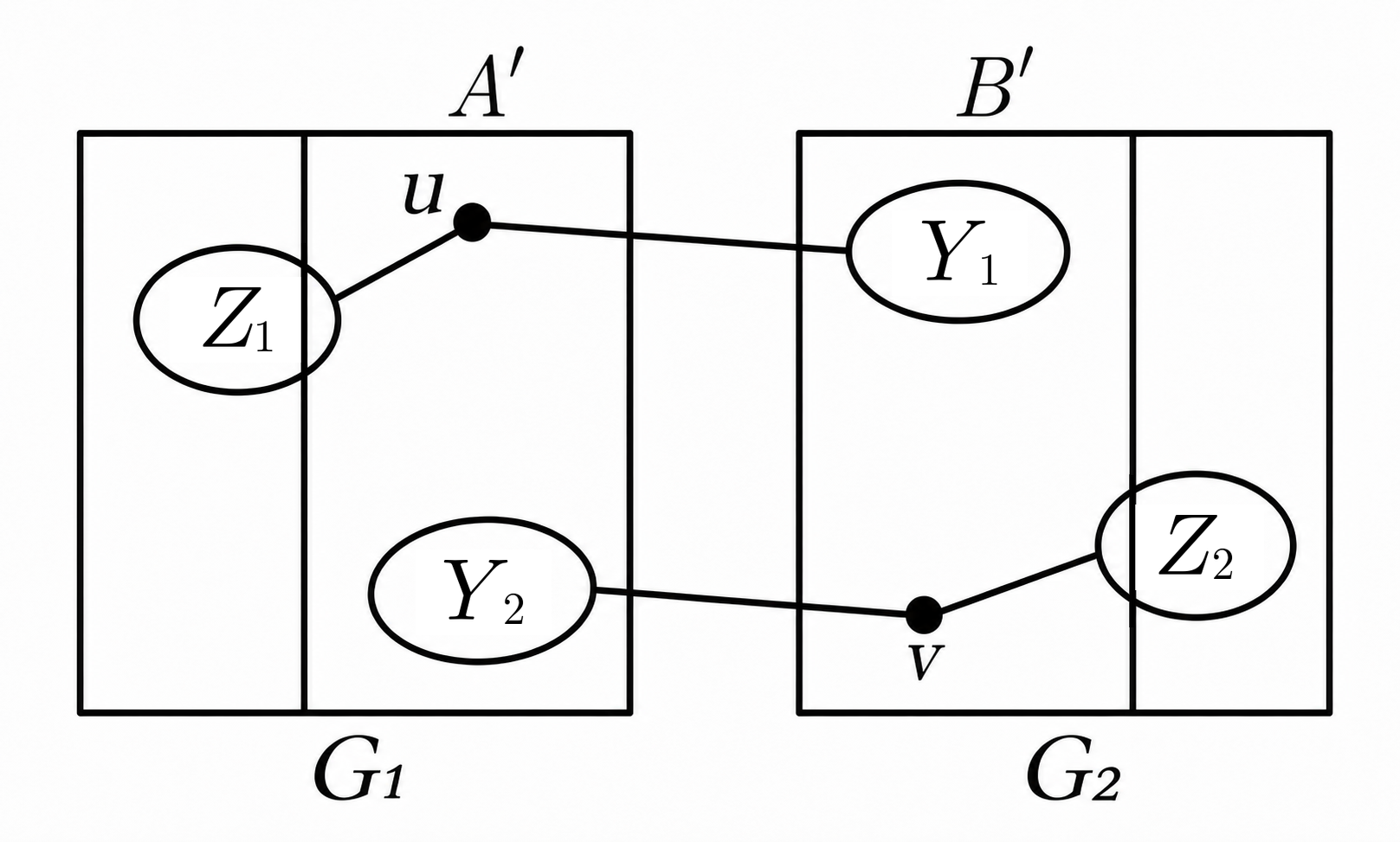}
      \caption{The illustration of $uv\notin E(G)$ and $x_1+x_2=0$}
		\label{fig:figure2.1}
	\end{figure}
	
	In this subcase,
	$d_G(u)+d_G(v)=y_1+y_2+z_1+z_2$.
	Since $u$ and $v$ are nonadjacent and have no common neighbor, the
	choices of $u$ and $v$ imply that $y_1\geq1$ and $y_2\geq1$.
	Counting the corresponding crossing edges gives
	$|S|\geq y_1+y_2+z_1y_1+z_2y_2$.
	It follows that
	\begin{align*}
		|S|-d_G(u)-d_G(v)+1
		&\geq y_1+y_2+z_1y_1+z_2y_2-z_1-z_2-y_1-y_2+1\\
		&=z_1(y_1-1)+z_2(y_2-1)+1\\
		&\geq1.
	\end{align*}
	Thus, $|S|\geq d_G(u)+d_G(v)-1$, which contradicts  Claim~1.
	
	\smallskip
	\noindent
	\textbf{Subcase 2.2.}
	$x_1+x_2>0$.
	
	\begin{figure}[H]
		\centering
		\includegraphics[width=0.5\textwidth]{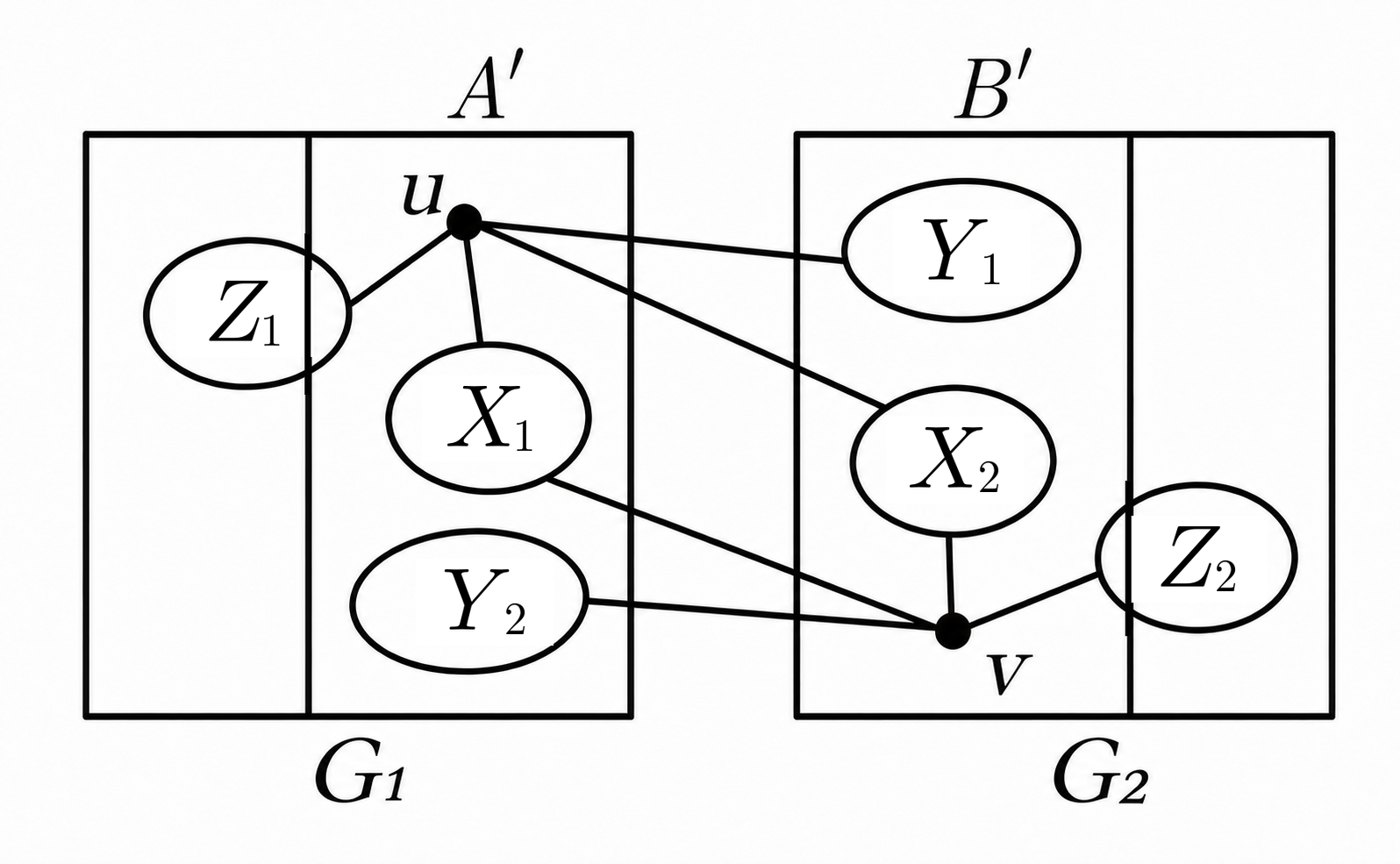}
       \caption{The illustration of $uv\notin E(G)$ and $x_1+x_2>0$}
		\label{fig:figure2.2}
	\end{figure}
	
	Once again,
	$d_G(u)+d_G(v)=2x_1+2x_2+y_1+y_2+z_1+z_2$.
	Counting the corresponding crossing edges of $G^2$ yields
	
	\begin{align*}
		|S|\geq{}&
		y_1+x_2+1+z_1(y_1+x_2)
		+x_1(x_2+y_1+1+z_2)
		+y_2(1+x_2+z_2)\\
		={}&
		y_1+y_2+x_1+x_2+z_1y_1+z_1x_2+x_1x_2+x_1z_2
		+y_2x_2+z_2y_2+x_1y_1+1.
	\end{align*}
	
	Therefore,
	\begin{align*}
		|S|-d_G(u)-d_G(v)+1
		\geq{}&
		(y_2-1)x_2+(y_1-1)x_1+(x_2+y_1-1)z_1+(x_1+y_2-1)z_2+x_1x_2+2.
	\end{align*}
	
	Since $x_1+x_2\neq0$, there are three possibilities.
	
	If $x_1\geq1$ and $x_2=0$, then the existence of a common neighbor in
	$V_1$ implies $y_1\geq1$. Hence
	\begin{align*}
		|S|-d_G(u)-d_G(v)+1
		\geq{}&
		(y_1-1)x_1+(y_1-1)z_1+(x_1+y_2-1)z_2+x_1x_2+2\geq2.
	\end{align*}
	
	If $x_2\geq1$ and $x_1=0$, then, symmetrically, $y_2\geq1$, and
	\begin{align*}
		|S|-d_G(u)-d_G(v)+1
		\geq{}&
		(y_2-1)x_2+(x_2+y_1-1)z_1+(y_2-1)z_2+x_1x_2+2\geq2.
	\end{align*}
	
	Finally, suppose that $x_1\geq1$ and $x_2\geq1$. Rearranging the
	preceding lower bound gives
	\begin{align*}
		|S|-d_G(u)-d_G(v)+1
		\geq{}&
		(x_2+y_1-1)z_1+(x_1+y_2-1)z_2+(x_1-1)(x_2-1)+y_1x_1+y_2x_2+1\geq1.
	\end{align*}
	Thus, $|S|\geq d_G(u)+d_G(v)-1$, which also contradicts Claim~1.
	
	All cases lead to contradiction. Therefore,
	$\lambda(G^2)=\delta(G^2)$, that is, $G^2$ is maximally edge-connected. $\Box$
	
  Sharpness of the bound in Theorem~\ref{thm:degree-sequence-square} is witnessed by the following family of graphs. Fix an integer $k\geq2$ and define
  $G_k=K_{k+1}+K_{k-1}+K_1+K_1+K_{k-1}+K_{k+1}$, where ``$+$'' denotes the sequential sum. For this construction, the order is $n=4k+2$, the minimum degree is $\delta(G_k)=d_n=k$, and $d_{n-\delta}=2k-1$. Hence,
  $d_{n-\delta}+d_n=3k-1=(3n-10)/4$.
  The square graph $G_k^2$ has minimum degree $2k$. On the other hand, the edges running between the third and fourth terms form a cut of size $2k-1$, and therefore
  $\lambda(G_k^2)\leq2k-1<2k=\delta(G_k^2)$. Consequently, the value $3k$ in Theorem~\ref{thm:degree-sequence-square} is best possible; the theorem no longer holds with $3k-1$ in its place.  Equivalently, the integer threshold determined by Theorem~4.3 is best possible.

	\section{Conclusions}
	
This paper concerns how edge-connectivity behaves under the square operation on graphs. First we establish a lower bound for $\lambda(G^2)$ in terms of  $\lambda(G)$, which extends the result in \cite{ChenXu} from range $4\leq\lambda(G)\leq7$ to $\lambda(G)\geq4$. Then we derive two conditions involving the degree sequence that guarantee $G^2$ is maximally edge-connected. In addition, examples show that  our two degree sequence conditions
are best possible.  These results enlarge the collection of known criteria for maximal edge-connectedness of square graphs and supplement the earlier work in \cite{BalbuenaDankelmann,ChenXu}. Further investigation may address whether the lower bound can be extended to higher graph powers and whether analogous criteria can be obtained from other graph parameters, including inverse degree \cite{DankelmannHellwigVolkmannInverse}, girth, diameter, and forbidden subgraphs \cite{SoneokaNakadaImasePeyrat}.

\end{document}